\documentclass[10pt,reqno,a4paper]{amsart}

\usepackage[pdftex]{graphicx}
\usepackage{a4wide}
\usepackage{amsmath}
\usepackage{amssymb,amsthm,hyperref,caption}%,times
\usepackage{color}
\definecolor{blau}{rgb}{0,0,0.75} %color for in-document links
\hypersetup{colorlinks,linkcolor=blau,citecolor=blue}

\allowdisplaybreaks

\newtheorem{theorem}{Theorem}
\newtheorem{lemma}{Lemma}
\newtheorem{coroll}{Corollary}
\newtheorem{prop}{Proposition}
\newtheorem{conj}{Conjecture}
\theoremstyle{definition}
\newtheorem{remark}{Remark}
\newtheorem{example}{Example}

\newcommand{\N}{\ensuremath{\mathbb{N}}}
\newcommand{\R}{\ensuremath{\mathbb{R}}}
\newcommand{\Z}{\ensuremath{\mathbb{Z}}}

\begin{document}

\author{Markus Kuba}
\address{Markus Kuba\\
Institut f{\"u}r Diskrete Mathematik und Geometrie\\
Technische Universit\"at Wien\\
Wiedner Hauptstr. 8-10/104\\
1040 Wien, Austria} \email{kuba@dmg.tuwien.ac.at}

\title{On embedded trees and lattice paths}

\begin{abstract}
Bouttier, Di Francesco and Guitter introduced a method for solving certain classes of algebraic recurrence relations arising the context of embedded trees and map enumeration. The aim of this note is to apply this method to three problems.
First, we discuss a general family of embedded binary trees, trying to unify and summarize several enumeration results for binary tree families, and also to add new results. Second, we discuss the family of embedded $d$-ary trees, embedded in the plane in a natural way.
Third, we show that several enumeration problems concerning simple families of lattice paths can be solved without using the kernel method by regarding simple families of lattice paths as degenerated families of embedded trees.
\end{abstract}

\thanks{The author was supported by the Austrian Science Foundation FWF, grant S9608-N13.}
%\subjclass[2000]{05C05} %

\maketitle

\section{Introduction}
Several families of embedded trees have been studied in the
literature. Binary trees, complete binary trees, several different
families of planar trees and more generally simply generated tree
families have been considered in a series of
papers~\cite{Boutt2003,Boutt2003II,Marck2004,Francesco2005,Bou2006,BouJan2006,Sva2006,Kne2006,DevSva2008,Alois2009,Drmota}:
it has been showed that embedded trees naturally arise in the
context of map enumeration and that properties of embedded trees are
closely related to a random measure called Integrated Superbrownian
Excursion. Combinatorial properties of embedded ternary trees where
studied using bijections between embedded ternary trees and
non-separable rooted planar maps~\cite{Schaeff1998,Left2000}, where
the authors studied a particular subclass of embedded ternary trees
named skew ternary trees~\cite{Schaeff1998}, or left ternary
trees~\cite{Left2000}, which are embedded ternary trees with no node
having label greater than zero. Using bijections between embedded
ternary trees with no label greater than zero and non-separable
rooted planar maps with $n + 1$ edges they obtained amongst others
an explicit result for the number of such trees of size $n$.
Some other enumerative results for embedded ternary trees
where derived in~\cite{Wo}. For the exact enumeration of embedded
trees and related problems Bouttier, Di Francesco and
Guitter~\cite{Boutt2003}, see also Di
Francesco~\cite{Francesco2005}, introduced a new method for solving
systems of recurrence relations. Bousquet-M\'elou~\cite{Bou2006} showed how this method can be used
to derive deep results about the enumeration of embedded binary
trees and families of embedded plane trees, and also about
properties of the Integrated Superbrownian Excursion. The aim of
this note is to continue the analysis of~\cite{Wo}. We use
generating functions and the method of~\cite{Boutt2003} to study a
general family of embedded binary trees, rederiving and unifying
several earlier results, and also the family of embedded $d$-ary
trees. Moreover, we show that some enumeration problems concerning simple families of
lattice paths, previously solved by Banderier and Flajolet~\cite{BandFla2002} using the kernel method, can be
treated using the method of~\cite{Boutt2003}. This work is
divided into three parts. The first part is devoted to the study of
a general family of binary trees embedded in the plane, summarizing
and rederiving a few of the enumerational results of~\cite{Boutt2003,Bou2006,Boutalk}.
The second part of this work is devoted to the study of embedded
$d$-ary trees. The third part is devoted to the enumeration of
lattice paths using the method of~\cite{Boutt2003,Francesco2005}, rederiving (and slightly extending) earlier results of Flajolet and
Banderier~\cite{BandFla2002}. Moreover, we use their method to (re-)derive 
other results. In particular, we derive the length
generating function of three vicious walkers and osculating walkers,
previously obtained earlier by Bousquet-M\'elou~\cite{oscu} using the kernel method, and Gessel.
In the next section we we recall some properties of the family of
$d$-ary trees and we discuss the (natural) embedding of $(2d+1)$-ary
trees and $(2d)$-ary trees into the plane. Section~\ref{Franc} is
devoted to a presentation of the method~\cite{Boutt2003,Francesco2005} following the exposition of Di
Francesco~\cite{Francesco2005}. Throughout this work we use the
notations $\N=\{1,2,\dots\}$, $\N_0=\{0,1,2,\dots\}$ and also
$\Z=\{\dots,-1,0,1,\dots\}$.

\section{The family of d-ary trees\label{DEEBNsecdary}}
The family of $d$-ary trees $\mathcal{T}$, with $d\ge 2$, can be described in a recursive way, which
says that a $d$-ary tree is either a leaf (an external node) or an internal node followed by $d$ ordered ternary trees, visually described by
the suggestive ``equation''
\begin{align*}
\label{DEEBaloissymb1}
\parbox{4cm}{
\unitlength 0.50mm
\linethickness{0.4pt}
\thicklines
\begin{picture}(78.33,22.00)
\put(64.33,1.67){\circle{6.57}}
%\put(64.33,4.67){\line(0,1){10}}
\put(62.00,4.00){\line(-1,1){11}}
\put(66.67,4.00){\line(1,1){11}}
\put(42.00,14.33){\makebox(0,0)[cc]{+}}
\put(21.00,13.50){\makebox(0,0)[cc]{\large{$\Box$}}}
\put(50.33,20.00){\makebox(0,0)[cc]{$\mathcal{T}$}}
%\put(64.33,15.00){\makebox(0,0)[cc]{$d$}}
\put(64.33,20.00){\makebox(0,0)[cc]{$\dots$}}
\put(78.33,20.00){\makebox(0,0)[cc]{$\mathcal{T}$}}
\put(1.67,14.67){\makebox(0,0)[cc]{$\mathcal{T}\quad=\quad$}}
\end{picture}}
\end{align*}
Here $\text{\small{$\bigcirc$}}$ is the symbol for an internal node and
$\text{\large{$\Box$}}$ is the symbol for a leaf or external node.
The generating function $T(z)=\sum_{n\ge 0}T_n z^n$ of the number of $d$-ary trees of size $n$ satisfies the equation
\begin{equation}
\label{DEEBteReqn0yi}
T(z)=1+zT^d(z),\quad\text{with}\,\,T(0)=1.
\end{equation}
Concerning the series expansion of the generating function $T(z)$ it
is convenient consider the shifted series $\tilde{T}(z):=T(z)-1$.
This corresponds to discarding external nodes (the empty tree) in
the description above; we obtain simply generated $d$-ary trees
$\tilde{\mathcal{T}}$ ), defined by the formal equation
\begin{equation}
   \label{DEEBteReqn1}
   \tilde{\mathcal{T}} = \bigcirc \times \varphi(\tilde{\mathcal{T}}),\quad\text{with}\,\, \varphi(t)=(1+t)^d,
\end{equation}
with $\bigcirc$ a node, $\times$ the cartesian product, and
$\varphi(\tilde{\mathcal{T}})$ the substituted structure. We refer to~\cite{MeirMoon1978} for the general definition of simply generated trees.
Let $T_n$ denote the number of ternary trees of size $n$, and $\tilde{T}_n$ the number of simply generated ternary trees of size $n$. By the formal description above~\eqref{DEEBteReqn1} the counting series $\tilde{T}(z)=\sum_{n\ge 1}\tilde{T}_nz^n$ satisfies
the functional equation
\begin{equation}
\label{DEEBteReqn2}
\tilde{T}(z)=z(1+\tilde{T}(z))^d,\quad \tilde{T}(0)=0.
\end{equation}
Due to the Lagrange inversion formula, see e.g.~\cite{Gould1983}, the number of $d$-ary trees of size $n$ is given by the so-called Fuss-Catalan numbers
$C_{n,d}=\frac1{(d-1)n+1}\binom{dn}{n}$,
\begin{equation}
\label{DEEBteReqn3}
\tilde{T}_n=[z^n]\tilde{T}(z)=\frac1{(d-1)n+1}\binom{dn}{n},\quad\text{and consequently}\quad \tilde{T}(z)=\sum_{n\ge 1}\binom{dn}{n}\frac{z^n}{(d-1)n+1}.
\end{equation}
Note that due to the definition the series $T(z)$ and $\tilde{T}(z)$ are related by $T(z)=\tilde{T}(z)-1$.

\subsection{Embedded d-ary trees\label{DEEBNssecembdary}}
By definition of $d$-ary trees each internal node with no children
has exactly $d$ positions to attach a new node, which are as usual
called external nodes or leaves, see Figure~\ref{DEEBNfig1}. We
embed $d$-ary trees in the plane by distinguishing between the cases
of even and odd $d$, respectively. Equivalently, we can distinguish
between $(2d+1)$-ary trees and $2d$-ary trees, with $d\ge 1$. The
root node has position zero. We recursively define the embedding of
$(2d+1)$-ary and $2d$-ary trees as follows. For $(2d+1)$-ary trees
an internal node with label/position $j\in\Z$ has exactly $(2d+1)$
children, being internal or external, placed at positions $\pm d$,
$\pm (d-1)$, $\dots$, $\pm 0$. For $2d$-ary trees an internal node
with label/position $j\in\Z$ has exactly $2d$ children, being
internal or external, placed at positions $\pm (2d-1)$, $\pm
(2d-3)$, $\dots$, $\pm 1$.  Following~\cite{Bou2006}, we call these
embedding \emph{natural embedding} of $d$-ary trees, because the
label a node is its abscissa in the natural integer embedding of the
tree.

\begin{figure}[htb]
\centering
\includegraphics[angle=0,scale=0.7]{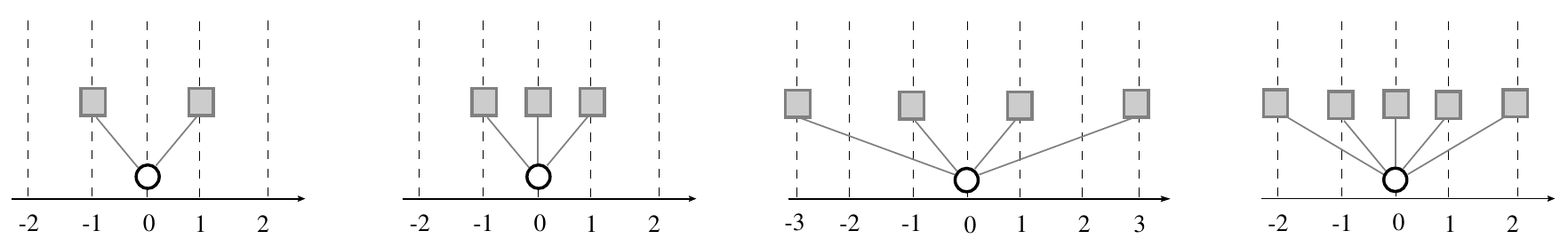}
\caption{Size one naturally embedded binary, ternary, quaternary and quinary trees together with their external nodes, i.e.~the possible increments $i\in\{\pm 1\}$, $i\in\{0,\pm 1\}$, $i\in\{\pm 1,\pm 3\}$ and $i\in \{0,\pm 1,\pm 2\}$.}
\label{DEEBNfig1}
\end{figure}

In this note we are interested in the number of embedded $d$-ary
trees having no label greater than $j$, with $j\in\N$. Let
$T_{j,2d+1}(z)$ and $T_{j,2d}(z)$ denote the generating function of
embedded $(2d+1)$-ary and $2d$-ary trees having no label greater
than $j$, $j\in\N$, with initial values
$T_{-1,2d+1}=T_{-2,2d+1}=\dots=T_{-d,2d+1}=1$ and
$T_{-1,2d+1}=T_{-2,2d+1}=\dots=T_{-2d+1,2d+1}=1$. Following the
observation of Bousquet-M\'elou we can think of $T_{j}(z)$ as the
generating function of embedded $d$-ary trees with \emph{root
labeled $j$}. By definition we obtain the following system of
recurrences for $T_j(z)$\footnote{Subsequently, we will usually drop
the subscripts $2d+1$ and $2d$ of $T_{j,2d+1}(z)$ and $T_{j,2d}(z)$
in order to simplify the presentation.}. For $(2d+1)$-ary trees we
get the system of recurrences
\begin{equation}
\label{DEEBNrec1a}
T_{j}(z)=1+z\prod_{\ell=-d}^{d}T_{\ell}(z),\quad j\ge 0,\quad\text{with}\,\,T_{-j}(z)=1,\quad\text{for}\,\,1\le j \le d,
\end{equation}
and for $2d$-ary trees we get the system of recurrences
\begin{equation}
\label{DEEBNrec1b}
T_{j}(z)=1+z\prod_{\ell=1}^{d}\Big(T_{2\ell-1}(z)T_{-2\ell+1}(z)\Big),\quad j\ge 0,\quad\text{with}\,\,T_{-j}(z)=1,\quad\text{for}\,\,1\le j \le 2d-1.
\end{equation}
Note that for both cases we have
\begin{equation*}
T_j(z)\to T(z)\quad\text{for}\quad j\to\infty,
\end{equation*}
in the sense of formal power series, where $T$ denotes the overall generating function~\eqref{DEEBteReqn0yi} of $(2d+1)$-ary and $2d$-ary trees, respectively. Note that this observation turns out to be crucial for the solution of the recurrence relation; see the original paper of Bouttier et al.~\cite{Boutt2003} and the next section. In the work~\cite{Wo} a different embedding for $2d$-ary trees is suggested. However, the embedding above for $2d$-ary trees turns out to be more easily analyzed and more natural, since the nodes are evenly placed in the plane.

\section{A method for solving infinite systems of algebraic recurrence relation \label{Franc}}
Bouttier, Di Francesco and Guitter introduced a method for solving certain classes of algebraic recurrence relations arising the context of embedded trees and map enumeration.
Our presentation of their method follows the exposition of Di Francesco~\cite{Francesco2005}. For a given integer $k\in\Z$ let $T_j(z)$, with $j\ge  k$, denote a family
of generating functions. Assume that the $T_j(z)$ satisfy algebraic
recurrence relations expressing $T_j(z)$ in terms of a finite number
of previous terms $T_{j-1}(z),T_{j-2}(z),\dots,T_{j-d}(z)$, with
$d\in\N$. The boundary data needed to entirely determine $T_j(z)$
should consist of $d$ consecutive initial values of $T_{j}(z)$.
Assume further that in the sense of formal power series $\lim_{j\to
\infty}T_j(z)$ exists, with $\lim_{j\to \infty}T_j(z)=T(z)$; note that $T(z)$ is also the solution of the unrestricted recurrence
relation for $T_j(z)$, holding for all $j\in\Z$. Exploiting the fact
that $\lim_{j\to \infty}T_j(z)=T(z)$ one uses the ansatz
$T_j(z)=T(z)(1-\rho_j(z))$, where $\rho_j(z)$ denotes an a priori
unknown formal power series with $\lim_{j\to\infty}\rho_j(z)=0$.
This allows to linearize the recurrence relations at large $j$, similar to first order asymptotic series expansion.

A first order expansion of the recurrence relation for $T_j(z)$ in terms of $\rho_j(z)$ leads to linear recurrence relations for $\rho_j(z)=\rho_j^{(1)}(z)$. It is readily solved using the classical ansatz $\rho_j^{(1)}(z)=\alpha \cdot X^j$, with unspecified $\alpha$. We can deduce that the general solution of the linearized recurrence relation is given by $\rho_j^{(1)}(z)=\sum_{\ell=1}^{d}\alpha_{\ell} \cdot X_{\ell}^j$,
where the $X_{\ell}(z)$, with $1\le \ell\le d$, are all solutions with modulus less one of the characteristic equation of the linear recurrence relation for the first order approximation $\rho_j(z)=\rho_j^{(1)}(z)$. In order to obtain the solution of the original problem one uses a full asymptotic series expansion
of the recurrence relation for $T_j(z)$ in terms of $\rho_j(z)=\sum_{n_1,\dots,n_d\ge 0}
\alpha_{n_1,\dots,n_d}\alpha_{n_1,\dots,n_k}\prod_{\ell=1}^{d}\big(X_{\ell}^{j}\big)^{n_{\ell}}$ and
compares order by order the contributions to the true solution. We recursively obtain the unspecified coefficients $\alpha_{n_1,\dots,n_k}$,
usually depending on $X_{\ell}(z)$, $1\le \ell \le d$, with free parameters $\alpha_{\mathbf{e}_{\ell}}$, where $\mathbf{e}_{\ell}$ denotes the $\ell$-th unit vector.

The \emph{main difficulty} is solve the recurrence relation for the coefficients $\alpha_{\mathbf{n}}=\alpha_{n_1,\dots,n_k}$.
Once these recurrence relations are solved, one can hopefully derive a compact expression for $\rho_j(z)$ and subsequently adapt
the unspecified parameters $\alpha_{\mathbf{e}_{\ell}}$, $1\le \ell \le d$, to the initial conditions $T_{j-1}(z),T_{j-2}(z),\dots,T_{j-d}(z)$.

\section{General families of embedded binary trees\label{SECbinaer}}
The family of ordinary (incomplete) binary trees $\mathcal{T}_1$, enumerated by the Catalan numbers, whose counting series $T=T(z)=\sum_{T\in\mathcal{T}_1}z^{|T|}$
satisfies the functional equation
\begin{equation*}
T(z)=1+zT(z)^2.
\end{equation*}
Bousquet-M\'elou~\cite{Bou2006} considered the embedding of this tree family in the plane according to
\begin{equation*}
T_j(z)=1+zT_{j-1}(z)T_{j+1}(z),\quad j\in\Z.
\end{equation*}
Here $T_j(z)$ denotes the generating function of a tree with root at position $j\in\Z$. Bouttier et al.~\cite{Boutt2003,Boutt2003II} and Bousquet-M\'elou~\cite{Bou2006} studied two families $\mathcal{T}_2$ and $\mathcal{T}_3$ of embedded plane trees which are closely related to families of maps. They can be realised as certain families of embedded binary trees.  Let $F=F(z)=\sum_{T\in\mathcal{T}_2}z^{|T|}$ and $G=G(z)=\sum_{T\in\mathcal{T}_3}z^{|T|}$ denote the counting series of the families
$\mathcal{T}_2$ and $\mathcal{T}_3$, satisfying the functional equations
\begin{equation*}
F(z)=\frac{1}{1-2zF(z)},\qquad G(z)=\frac{1}{1-3zG(z)},
\end{equation*}
or equivalently,
\begin{equation*}
F(z)=1+2zF(z)^2,\qquad G(z)=1+3zG(z)^2.
\end{equation*}
These tree families are embedded according to
\begin{equation*}
F_j(z)=\frac{1}{1-z\big(F_{j-1}(z)+F_{j+1}(z)\big)},\quad
G_j(z)=\frac{1}{1-z\big(G_{j-1}(z)+G(z) +G_{j+1}(z)\big)}, j\in\Z,
\end{equation*}
or equivalently by
\begin{equation*}
F_j(z)=1+zF_j(z)\big(F_{j-1}(z)+F_{j+1}(z)\big),\quad G_j(z)=1+zG_j(z)\big(G_{j-1}(z)+G_j(z)+G_{j+1}(z)\big).
\end{equation*}
For the three tree families $\mathcal{T}_1$, $\mathcal{T}_2$ and $\mathcal{T}_3$ it was shown that the generating functions of trees with small labels, i.e.~tree in which all labels are less or equal $j$, are algebraic and explicit expressions were obtained.

\subsection{Embedding of a general family of binary trees}
We discuss properties of the family $\mathcal{T}$ of weighted binary trees, defined according to a functional equation for its counting series $T=T(z,v_1,v_2,w_1,w_2,w_3)=\sum_{G\in\mathcal{T}}z^{|G|}$,
\begin{equation*}
T=1+z(2v_1+v_2)T+z(w_1+w_2+2w_3)T^2.
\end{equation*}
We can interpret $v_1,v_2,w_1,w_2,w_3$ either as weights, $v_1,v_2,w_1,w_2,w_3\ge 0$, or as variables encoding different kinds of nodes, which would lead to a refined enumeration of trees. Concerning the second point of view one could for example consider $[z^nv_1^{m_1}v_2^{m_2}w_1^{\ell_1}w_2^{\ell_2}w_3^{\ell_3}]T$, with $m_1+m_2+\ell_1+\ell_2+\ell_3=n$.
By solving the quadratic equation for $T$ one easily obtains the following explicit result.
\begin{equation}
\label{DEEBNbinaer1}
T=\frac{1-z(2v_1+v_2)-\sqrt{\big(1-z(2v_1+v_2)\big)^2-4z(w_1+w_2+2w_3)}}{2z(w_1+w_2+2w_3)}.
\end{equation}
We reobtain the previously considered families and several other tree families, binary and non-binary, by suitable sometimes non-unique choices of $v_1,v_2$ and $w_1,w_2,w_3$.
\begin{example}
Binary trees (Catalan numbers) \textbf{A000108} are obtained by setting $v_1=v_2=w_2=w_3=0$ and $w_1=1$,
the number of rooted Eulerian edge maps in the plane \textbf{A052701} are obtained by setting $v_1=v_2=w_1=w_2=0$ and $w_3=1$,
Blossom trees or equivalently rooted planar maps \textbf{A005159} are obtained by setting $v_1=v_2=w_1=0$ and $w_2=w_3=1$,
Schr\"oder trees (large Schr\"oder numbers) \textbf{A006318} can be obtained setting $v_2=w_2=w_3=0$ and $v_1=w_1=1$,
planar rooted trees with tricolored end nodes \textbf{A047891} can be obtained setting $v_1=w_2=w_3=0$ and $v_2=w_1=1$,
the choice $v_1=v_2=w_1=1$ and $w_2=w_3=0$ gives sequence \textbf{A082298},
the choice $v_1=w_3=1$ and $v_2=w_1=w_2=0$ gives sequence \textbf{A103210}; several other sequences in Sloane's Encyclopedia~\cite{Sloane} can be obtained by suitable choices of the parameters.
\end{example}

We embed this family according to the following recurrence relation
for $T_j=T_j(z,v_1,v_2,w_1,w_2,w_3)$.
\begin{equation}
\label{DEEBNgenrec1}
T_j=1+z\big(v_1T_{j-1}+ v_1T_{j+1} +v_2T_j \big)+ z\Big(w_1T_{j-1}T_{j+1}+w_2T_{j}^2 + w_3T_{j}\big(T_{j-1}+T_{j+1}\big)\Big),
\end{equation}
with $j\in\Z$. We will see that we can reobtain the previously discussed families $\mathcal{T}_1$, $\mathcal{T}_2$ and $\mathcal{T}_3$ and their counting series by the following choices of the weights/variables $v_1,v_2,w_1,w_2,w_3$: $T(z,0,0,w_1,0,0)$, $T(z,0,0,0,w_2,w_2)$ and $T(z,0,0,0,0,w_2)$.

\smallskip

We will show that for several choices of the weights $w_j$ and arbitrary weights $v_i$ the generating functions of trees with small labels in the embedded family $\mathcal{T}$, i.e.~tree in which all labels are less or equal $j$, can be explicitly obtained.

\subsection{Trees with small labels}
Our starting point is the recurrence relation below for $T_j$.
\begin{equation}
\begin{split}
\label{DEEBNeqn1}
T_j=1+z\big(v_1T_{j-1}+ v_1T_{j+1} +v_2T_j \big)+ z\Big(w_1T_{j-1}T_{j+1}+w_2T_{j}^2 + w_3T_{j}\big(T_{j-1}+T_{j+1}\big)\Big),
\end{split}
\end{equation}
for $j\ge 0$ with initial value given by $T_{-1}=1$ or $T_{-1}=0$,
depending on particular counting problem,
see~\cite{Boutt2003,Boutt2003II,Francesco2005,Bou2006}.
Following the approach presented in Section~\ref{Franc} we use that fact that for $j$
tending to infinity we have $T_j\to T$ in the sense of formal power
series, with $T$ given by~\eqref{DEEBNbinaer1}. We make the
\emph{ansatz} $T_j=T(1-\rho_j)$, where $T=T(z,v_1,v_2,w_1,w_2,w_3)$
denotes the generating function of the family $\mathcal{T}$ defined
by~\eqref{DEEBNbinaer1}, with $\rho_j\to 0$ as $j$ tends to
infinity. We expend Equation~\ref{DEEBNeqn1} with respect to the
ansatz and compare the terms tending at a similar rate to zero in
the asymptotic expansion of $T_j$ as $j$ tends to infinity, neglecting terms $\rho_j^2$, $\rho_j\rho_{j+1}$ and $\rho_j\rho_{j-1}$. We get
the linearized equation
\begin{equation*}
-T\rho_j=-zT\big( v_1(\rho_{j-1}+\rho_{j+1})+v_2\rho_j\big) - zT^2\big(w_1(\rho_{j-1}+\rho_{j+1}) + 2w_2\rho_j + w_3(\rho_{j-1}+2+\rho_j+\rho_{j+1})\big).
\end{equation*}
Now we make a refined ansatz $\rho_j=X^j$ in order to solve this linear recurrence relation for $\rho_j$, assuming that $X$ is a formal power series depending on variables/weights
$z,v_1,v_2,w_1,w_2,w_3$ with $|X|<1$. We obtain the so-called characteristic equation for the series $X$,
\begin{equation}
\label{DEEBNeqn2}
1=z\Big( v_1\big(\frac1X+X\big)+v_2\Big) + zT\Big(w_1(\frac1X+X\big) + 2w_2 + w_3\big(\frac1X+2+X\big)\Big).
\end{equation}
We observe that $X$ is a power series in $z,v_1,v_2,w_1,w_2,w_3$ and has non-negative coefficients. Consequently, the proper solution is given by
\begin{equation}
\label{DEEBNeqnX}
X=\frac{1-z(v_2+2T(w_2+w_3))-\sqrt{\big(1-z(v_2+2T(w_2+w_3))\big)^2-4z^2(v_1+T(w_1+w_3))^2}}{2z(v_1+T(w_1+w_3))}.
\end{equation}
One readily checks that the expression above for $X$ is indeed a power series in $z,v_1,v_2,w_1,w_2,w_3$ and has non-negative coefficients.
Using the definition of the series $T$ we can express $T$ solely in terms of the series $X$
\begin{equation*}
T=\frac{t_1(X)+\sqrt{t_1(X)^2 +4t_2(X)\big(v_1(1+X^2)+v_2X\big)}}{t_2(X)},
\end{equation*}
with respect to the polynomials $t_1(X)=t_1(v_1,w_1,w_2,w_3,X)$ and $t_2(X)=t_2(w_1,w_2,w_3,X)$ defined by
\begin{equation*}
t_1(X)=w_1(1+X^2)+2w_2X+w_3(1+X)^2-v_1(1-X)^2,\qquad t_2(X)=w_1(1-X+X^2)+w_2X+w_3(1+X^2).
\end{equation*}
We make the more refined ansatz $\rho_j=\sum_{i\ge
1}\alpha_i(X^{j})^{i}$, with unspecified $\alpha_1$ and
$\alpha_i=\alpha_i(X)$, which amounts to an asymptotic expansion of
$\rho_j$ for $j$ tending to infinity. Next we compare the terms with
the same order of magnitude in~\eqref{DEEBNeqn1} as $j$ tends
infinity. We obtain from~\eqref{DEEBNeqn1}, using the
relation~\eqref{DEEBNeqn2}, the following recurrence relation for
$\alpha_{n+1}$, with $n\ge 0$.
\begin{equation}
\label{DEEBNeqn3}
\alpha_{n+1}\Big(\frac{v_1}{T}+w_1+w_3\Big)\Big(\frac{1}{X^{n+1}}+X^{n+1}-\frac{1}{X}-X\Big)
=\sum_{i=1}^{n}\alpha_i\alpha_{n+1-i}\Big(w_1X^{n+1-2i} +w_2 +w_3\Big(\frac{1}{X^{i}}+X^{i}\Big)\Big).
\end{equation}
We observe that the variable $v_2$ only appears in the defining equations for series $T$ and $X$, but not in the recurrence relation for $\alpha_{n}$.
Introducing the quantity $\beta_{n+1}=\alpha_{n+1}\big(\frac{v_1}{T}+w_1+w_3\big)^n$, $n\ge 0$, we obtain the simplified recurrence relation
\begin{equation}
\label{DEEBNeqn4}
\beta_{n+1}\Big(\frac{1}{X^{n+1}}+X^{n+1}-\frac{1}{X}-X\Big)
=\sum_{i=1}^{n}\beta_i\beta_{n+1-i}\Big(w_1X^{n+1-2i} +w_2 +w_3\Big(\frac{1}{X^{i}}+X^{i}\Big)\Big).
\end{equation}
Let $f(t)$ denote the formal power series $f(t)=\sum_{n\ge 1}\beta_n t^n$. Equation~\ref{DEEBNeqn4} is equivalent to a functional equation
for $f(t)$:
\begin{equation*}
f(tX)+f\Big(\frac{t}{X}\Big) - \Big(\frac{1}{X}+X\Big)f(t) = w_1f(tX)\cdot f\Big(\frac{t}{X}\Big)
+w_2f(t)^2 +w_3f(t)\cdot \Big( f(tX)+f\Big(\frac{t}{X}\Big)  \Big).
\end{equation*}
One already knows the solutions of Equation~\ref{DEEBNeqn4} in the cases $w_2=w_3=0$, see Bousquet-M\'elou~\cite{Bou2006},
and $w_1=0$, $w_1=w_2=0$, see Bouttier et al.~\cite{Boutt2003} and also~\cite{Bou2006}.
We will provide the solution of Equations~\ref{DEEBNeqn3} and~\ref{DEEBNeqn3}, respectively, in the case $w_1$ and $w_2=w_3$, excluding the degenerate case $w_1=w_2=w_3=0$.
\begin{lemma}
\label{DEEBNlem1}
For given parameters $w_1$ and $w_2=w_3$, excluding the degenerate case $w_1=w_2=w_3=0$, the solution $\alpha_{n}$ of the recurrence relation~\ref{DEEBNeqn3} is for $n\ge 1$ given by
\begin{equation*}
\alpha_n= \frac{X^{n-1}\alpha_1^n \big(w_1 X + w_2(1+X+X^2)\big)^{n-1}(1-X^n)}{\big(\frac{v_1}{T}+w_1+w_2\big)^{n-1}(1-X)^{2n-1}(1+X+X^2)^{n-1}(1+X)^{n-1}}.
\end{equation*}
\end{lemma}
We could not solve directly the functional equation for $f(t)$. Instead we obtained the solution in an \emph{experimental way} using the computer algebra software \texttt{Maple}. Once the solution of the recurrence relation is guessed,
it is readily rigourously checked that it satisfies the recurrence relation~\eqref{DEEBNeqn3}, or equivalently that the generating function
$f(t)=\sum_{n\ge 1}\beta_n t^n$ satisfies the stated functional equation.
Unfortunately, we could not solve the recurrence relation in full generality $w_2\neq w_3$, except for the already known special case $w_1=w_2=0$ and $w_3\neq 0$~\cite{Boutt2003,Bou2006}; it is given by
\begin{equation*}
\alpha_n= \frac{X^{n-1}\alpha_1^n w_3^{n-1}(1-X^{2n})}{\big(\frac{v_1}{T}\big)^{n-1}(1-X)^{2n-1}(1+X+X^2)^{n-1}(1+X)}.
\end{equation*}
However, the result of Lemma~\ref{DEEBNlem1} already covers and generalizes the result for two previously treated families, the cases
$v_1=v_2=w_2=0$ of binary trees and $v_1=v_2=w_1=0$ of a family of planar trees, which we interpret as embedded binary trees.
It seems that the structure of the values $\alpha_n$ is not regular in the other cases. We performed some computer experiments and
we state the following conjecture on the values of $\alpha_{n}$ for $w_2=0$ and $w_1=w_3=1$.
\begin{conj}
In the case $w_2=0$ and $w_1=w_3=1$ the solution $\alpha_{n}$ of the recurrence relation~\ref{DEEBNeqn3} is given by
\begin{equation*}
\alpha_n=\frac{\alpha_1^n X^{n-1}p_n(X)}{\big(\frac{v_1}{T}+2\big)^{n-1}(1-X)^{2n-2}(1+X)^{2\left\lfloor\frac{n-1}{2}\right\rfloor}(1+X^2)^{\left\lfloor\frac{n-1}{2}\right\rfloor}},
\end{equation*}
where the sequence of polynomials $(p_n(X))_{n\in\N}$ with initial values
\begin{equation*}
p_1(X)=1,\quad p_2(X)=1,\quad p_3(X)=X^4+2X^3+2X+1,
\end{equation*}
is for $n\ge 2$ recursively defined by
\begin{equation*}
\begin{split}
p_{2n}(X)=P_{2n-1}(X)-2X^2p_{2n-2}(X),\quad p_{2n+1}(X)=p_{n+2}(X)p_{n+1}(X)-4X^4p_{n}(X)p_{n-1}(X).
\end{split}
\end{equation*}
\end{conj}
We return to our previous case of $w_1$ and $w_2=w_3$.
In order to simplify the presentation we set
\begin{equation*}
\alpha_1= \frac{\big(\frac{v_1}{T}+w_1+w_2\big)(1-X^2)(1-X^3)}{\big(w_1X+w_2(1+X+X^2)\big)}\cdot\lambda,\quad\text{with}
\quad\lambda=\lambda(X,v_1,w_1,w_2),
\end{equation*}
and obtain the following result.
\begin{theorem}
\label{DEEBNthe1}
For given parameters $w_1$ and $w_2=w_3$, excluding the degenerate case $w_1=w_2=w_3=0$, a solution of the recurrence relation~\ref{DEEBNgenrec1} 
with free parameter $\lambda$ is given by
\begin{equation*}
T_j=T\cdot\Big(1-\frac{\big(\frac{v_1}{T}+w_1+w_2\big)\lambda(1-X^2)(1-X^3)X^j}{\big(w_1X+w_2(1+X+X^2)\big)(1-\lambda X^{j+1})(1-\lambda X^{j+2})}\Big),
\end{equation*}
with series $X$ given by~\eqref{DEEBNeqnX}.
\end{theorem}
Now we can easily reobtain the previous results of~\cite{Boutt2003,Bou2006} by suitable choices of $v_1,v_2,w_1,w_2,w_3$ and adapting $\lambda$ to the initial value $T_{-1}$. The quadratic equation relating $\lambda$ and $T_{-1}$ normally has two distinct solutions; we use the fact
that $T_j$ a priori has a power series expansion at $z=0$ to identify the right solution.
\begin{coroll}[\cite{Boutt2003,Bou2006}]
In the case of embedded binary trees, $v_1=v_2=w_2=w_3=0$ and $w_1=1$ with $T_{-1}=1$, we reobtain the result
\begin{equation*}
T_j=T\cdot\frac{(1-X^{j+2})(1-X^{j+7})}{(1- X^{j+4})(1- X^{j+5})},\quad j\ge -1.
\end{equation*}
In the case of embedded planar trees, $v_1=v_2=w_1=0$ and $w_2=w_3=1$ with $T_{-1}=0$, we reobtain the result
\begin{equation*}
T_j=T\cdot\frac{(1-X^{j+1})(1-X^{j+4})}{(1- X^{j+2})(1-X^{j+3})},\quad j\ge -1.
\end{equation*}
\end{coroll}
\begin{remark}
As mentioned above one can readily obtain numerous enumerative results from Theorem~\ref{DEEBNthe1}. The solutions turn out
to be usually more involved due to the adaption to initial values $T_{-1}=1$ or $T_{-1}=0$.
\end{remark}

\subsection{The height of planar trees}
A more general form of recurrence relation~\eqref{DEEBNeqn1} reads the following way.
\begin{equation*}
T_j=1+z\big(v_1T_{j-1}+v_2T_j+ v_3T_{j+1}  \big)+ z\Big(w_1T_{j-1}T_{j+1}+w_2T_{j}^2 + w_3T_{j}T_{j-1}+w_4T_jT_{j+1}\big)\Big).
\end{equation*}
It seems very difficult to obtain solutions for this recurrence relation.
However, there exist a subclass $v_3=w_1=w_2=w_4=0$, setting for the sake of simplicity $w_3=1$, which is explicitly solvable
\begin{equation}
\label{DEEBNknuth1}
T_j=1+z\big(v_1T_{j-1}+v_2T_j\big)+ zT_{j-1}T_{j},\quad\text{for}\quad j\ge 1.
\end{equation}
This was observed earlier by Bousquet-M\'elou~\cite{Boutalk}. This subclass is of particular importance due to the connection
with the height of plane trees~\cite{Knuth}, corresponding to the case $v_1=v_2=0$, with initial condition $T_0(z)=1$.
By the approach presented in Section~\ref{Franc} we use the fact that
\begin{equation*}
T_j\to T,\quad\text{with}\quad T=1+z(v_1+v_2)T+zT^2,
\end{equation*}
and the ansatz $T_j=T(1-\rho_j)$. This leads to
\begin{equation*}
-\rho_j T = -zT \big(v_1\rho_{j-1}+v_2\rho_{j}\big) - zT^2\big(\rho_{j-1}+\rho_{j}\big) + zT^2\rho_{j-1}\rho_{j}.
\end{equation*}
For the first order expansion we consider the terms tending at the same rate to zero as $j$ tends to infinity, neglecting the term $\rho_{j-1}\rho_{j}$. We get the linearized equation
\begin{equation*}
\rho_j =z \big(v_1\rho_{j-1}+v_2\rho_{j}\big)  + zT\big(\rho_{j-1}+\rho_{j}\big).
\end{equation*}
This recurrence relation is readily solved by $\rho_j=X^j$, with $X=X(z)$ given by
satisfying
\begin{equation*}
X=\frac{z(v_1+T)}{1-z(v_2+T)}.
\end{equation*}
As before he more refined ansatz $\rho_j=\sum_{i\ge
1}\alpha_i(X^{j})^{i}$, with unspecified $\alpha_1$ and
$\alpha_i=\alpha_i(X)$ leads to a recurrence relation for $\alpha_{n+1}$,
\begin{equation*}
\alpha_{n+1}\Big(\frac{v_1}{T}+1\Big)\frac{1-X^{n}}{X^{n+1}}=\sum_{i=1}^{n}\alpha_{i}\alpha_{n+1-i}\frac1{X^i},\quad n\ge 0.
\end{equation*}
Proceeding as before, we introduce the quantity $\beta_n=\alpha_n(\frac{v_1}{T}+1)^{n-1}$, $n\ge 1$ and solve the arising recurrence relation in an  experimental way using the computer algebra software \texttt{Maple}.
We obtain the solution
\begin{equation*}
\alpha_{n}=\frac{\alpha_1^{n}X^{n-1}}{\Big(\frac{v_1}{T}+1\Big)^{n-1}(1-X)^{n-1}},\quad n\ge 1,
\end{equation*}
with unspecified $\alpha_1$. In order to simplify the presentation we set $\alpha_1=\lambda(1-X)/(\frac{v_1}{T}+1)$ and obtain the following result.
\begin{theorem}[\cite{Boutalk}]
\label{DEEBNtheknuth}
A solution of the recurrence relation~\ref{DEEBNknuth1} with free parameter $\lambda$ is given by
\begin{equation*}
T_j=T\cdot\Big(1-\frac{(\frac{v_1}{T}+1)(1-X)\lambda X^j}{(1-\lambda X^{j+1})}\Big),
\end{equation*}
with $X=z(v_1+T)/(1-z(v_2+T))$.
\end{theorem}
Setting $v_1=v_2=0$ and adapting to the initial condition $T_0=1$ gives the following result.
\begin{coroll}[\cite{Knuth,Boutalk}]
The generating function of plane trees of height $\le j$ is given by
\begin{equation*}
T_j=T\cdot \frac{1-X^{j+1}}{1-X^{j+2}},\quad\text{with}\quad X=zT/(1-zT)=T-1.
\end{equation*}
\end{coroll}
Note that the free parameter $\lambda$ in Theorem~\ref{DEEBNtheknuth} allows to obtain refined enumerations, as shown in~\cite{Boutalk}.

\subsection{A family of ternary trees}
We have seen that one can eliminate the variables $v_1$ and $v_2$ from the recurrence relation~\eqref{DEEBNeqn3} (and also in~\eqref{DEEBNknuth1}) for $\alpha_n$ by a proper substitution, leading to the simplified recurrence relation~\eqref{DEEBNeqn4} for the values $\beta_n$. This is not the case anymore even for ternary trees. Consider for example the family $\mathcal{T}$ of weighted ternary trees, defined according to a functional equation for its counting series $T=T(z,v_1,v_2)=\sum_{T\in\mathcal{T}}z^{|T|}$,
\begin{equation*}
T(z)=1+z(2v_1+v_2)T(z)+zT(z)^3,
\end{equation*}
embedded according to recurrence relation
\begin{equation*}
T_j(z)=1+z\big(v_1T_{j-1}+ v_1T_{j+1}(z) +v_2T_j(z) \big) + zT_{j-1}(z)T_j(z)T_{j+1}(z),
\end{equation*}
with $j\in\Z$.
Proceeding as before, i.e.~making the ansatz $T_j=T(1-\rho_j)$ and subsequent refinements
$\rho_j=\sum_{n\ge 1}\alpha_n(X^j)^n$ with $X$ being the solution of
\begin{equation*}
1=z\Big( v_1\big(\frac1X+X\big)+v_2\Big) +zT^2\Big(\frac{1}{X}+1+X\Big),
\end{equation*}
with $|X|<1$, one obtains the recurrence relation
\begin{equation*}
\begin{split}
\alpha_{n+1}\Big(\frac{v_1}{T^2}+1\Big)\Big(\frac{1}{X^{n+1}}+X^{n+1}-\frac{1}{X}-X\Big)
&=\sum_{i=1}^{n}\alpha_i\alpha_{n+1-i}\Big(X^{n+1-2i} +\frac{1}{X^{i}}+X^{i}\Big)\\
&\quad+\sum_{i_1+i_2+i_3=n}\alpha_{i_1}\alpha_{i_2}\alpha_{i_3}X^{i_1-i_3}.
\end{split}
\end{equation*}
Unfortunately, we are not able to solve this recurrence relation for $v_1\neq 0$. We observe that as before the variable $v_2$ only appears in the defining equations for series $T$ and $X$, but not in the recurrence relation for $\alpha_n$. In the case $v_1=0$ we can use the solution of~\cite{Wo}, and subsequently may obtain a refinement of a result of~\cite{Wo}.

\section{Embedded (2d+1)-ary trees with small labels}
The starting point of our considerations is the system of recurrences~\eqref{DEEBNrec1a}.
We make the ansatz $T_j(z)=T(1-\rho_j)$, with $\rho_j=\rho_j(z)\to 0$ as $j$ tends to infinity, for $z$ near zero.
We expend Equation~\ref{DEEBNrec1a} with respect to the ansatz and obtain
\begin{equation*}
T(1-\rho_j)=1+zT^{2d+1}\prod_{\ell=-d}^{d}(1-\rho_{j+\ell}).
\end{equation*}
By definition of $(2d+1)$-ary trees~\eqref{DEEBteReqn0yi} we have $1=T-zT^{2d+1}$. Consequently,
\begin{equation*}
T(1-\rho_j)=T\Big(1-zT^{2d}+zT^{2d}\prod_{\ell=-d}^{d}(1-\rho_{j+\ell})\Big).
\end{equation*}
The equation above is equivalent to
\begin{equation*}
1-\rho_j=1-zT^{2d}+zT^{2d}\prod_{\ell=-d}^{d}(1-\rho_{j+\ell}).
\end{equation*}
By expansion of the product on the right hand side of the equation above we obtain the \emph{main equation}
\begin{equation}
\label{DEEBNmaster}
\rho_j=zT^{2d}-zT^{2d}\prod_{\ell=-d}^{d}(1-\rho_{j+\ell})=zT^{2d}\sum_{\ell=1}^{2d+1}(-1)^{\ell-1}\sum_{\mathbf{b}_{\ell}\subseteq \{-d,\dots,d\}}\bigg(\prod_{k=1}^{\ell}\rho_{j+b_k}\bigg),
\end{equation}
with $\mathbf{b}_{\ell}=\{b_1,\dots,b_{\ell}\}$ running over all subset of $\{-d,\dots,d\}$ of size $\ell$, $1\le \ell\le 2d+1$.
Comparing the terms tending at a similar rate to zero as $j$ tends to infinity we obtain the linear recurrence relation
\begin{equation*}
\rho_j=zT^{2d}\sum_{\ell=-d}^{d}\rho_{j+\ell}.
\end{equation*}
An ansatz $\rho_j=\alpha X^j$, assuming that there exists a formal power series $X=X(z)$ with $|X|<1$ for $z$ near 0, leads to the so-called \emph{characteristic equation}
\begin{equation}
\label{DEEBNchar2d1}
1=zT^{2d}\sum_{\ell=-d}^{d}X^{\ell},\quad\text{or equivalently}\quad 1=Z\sum_{\ell=-d}^{d}X^{\ell},\quad \text{with}\,\,
Z:=zT^{2d}.
\end{equation}
The equation above is identical to the characteristic equation of lattice path
with step set $\mathcal{S}=\{(1,\ell) \mid -d\le \ell \le d\}$, see Banderier and Flajolet~\cite{BandFla2002}.
We can use the very general considerations of~\cite{BandFla2002} summarized below in Lemma~\ref{DEEBNlemFla1} concerning such equations.
\begin{lemma}[Banderier and Flajolet~\cite{BandFla2002}]
\label{DEEBNlemFla1}
Let $\mathcal{S}\subseteq \Z$ denote a non-empty finite subset of the integers $\mathcal{S}=\{b_1,\dots,b_{m}\}$
and $\Pi:=\{w_1,\dots,w_{m}\}\subseteq \R^{+}$ the set of associated weights.
The characteristic polynomial $P(X)$ associated to $\mathcal{S}$ and $\Pi$ is a Laurent polynomial in $X$
given by $P(X)=\sum_{\ell=1}^{m}w_{\ell}X^{b_{\ell}}$. Let $c=-\min\{b_\ell\}$ and $d=\max\{b_\ell\}$.
The characteristic equation associated to $\mathcal{S}$ is given by
\begin{equation*}
1-ZP(X)=0,\quad\text{or equivalently}\quad X^c - Z X^cP(X)=0.
\end{equation*}
For $Z$ near zero the characteristic equation has $c+d$ solutions, of which $c$ solutions $X_1(Z),\dots X_c(Z)$ are small solution
with $|X_{\ell}(Z)|<1$ for $Z$ near zero. These $c$ so-called small branches are conjugate of each other at $Z=0$: there exist two functions $A$ and $B$ analytic at $Z=0$ and nonzero there such that in a neighbourhood of zero one has
\begin{equation*}
X_{\ell}=X_\ell(Z)=\omega^{\ell-1}Z^{1/c}A\big(\omega^{\ell-1}Z^{1/c}\big),
\quad\text{with}\quad  \omega=e^{2i\pi/c},
\end{equation*}
where $i$ denote the imaginary unit $i^2=-1$.
\end{lemma}
The result of Banderier and Flajolet~\cite{BandFla2002} was initially derived in the context of the enumeration of (weighted) lattice paths.
We apply Lemma~\ref{DEEBNlemFla1} to case $\mathcal{S}=\{\ell \mid -d\le \ell \le d\}$, $w_{\ell}=1$, $-d\le \ell \le d$ and $Z=zT^{2d}$.
Consequently the equation $1=Z\sum_{\ell=-d}^{d}X^{\ell}$, with $Z=zT^{2d}$, has $d$ small solution $X_1(z),\dots X_d(z)$.
We refine the previous ansatz $\rho_j=\alpha X^j$ in terms of the $d$ solutions $X_1(z),\dots X_d(z)$ of the characteristic equation in the following way.
\begin{equation}
\label{DEEBNrefined}
\rho_j=\sum_{n_1,\dots,n_d\ge 0}\alpha_{n_1,\dots,n_d}X_1^{jn_1}\dots X_d^{jn_d}
=\sum_{\mathbf{n}\ge \boldsymbol{0}}\alpha_{\mathbf{n}}\mathbf{X}^{j\mathbf{n}},
\end{equation}
with $\alpha_{0,0,\dots,0}=0$ and unspecified initial values $\alpha_{\mathbf{e}_{\ell}}$, $1\le \ell\le d$, where $\mathbf{e}_{\ell}$ denotes the $\ell$-th unit vector.
By the refined ansatz the main equation~\eqref{DEEBNmaster} reads the following way.
\begin{equation*}
\sum_{\mathbf{n}\ge \boldsymbol{0}}\alpha_{\mathbf{n}}\mathbf{X}^{j\mathbf{n}}=zT^{2d}\sum_{\ell=1}^{2d+1}(-1)^{\ell-1}\sum_{\mathbf{b}_{\ell}\subseteq \{-d,\dots,d\}}\prod_{k=1}^{\ell}
\bigg(\sum_{\mathbf{n}\ge \boldsymbol{0}}\alpha_{\mathbf{n}}\mathbf{X}^{(j+b_k)\mathbf{n}}\bigg).
\end{equation*}
Our goal is to determine the unknown coefficients $\alpha_{\mathbf{n}}=\alpha_{\mathbf{n}}(\mathbf{X})$ as functions of $X_1,\dots,X_d$ and the unspecified initial values $\alpha_{\mathbf{e}_{\ell}}$, $1\le \ell \le d$.
In order to do so we compare the terms with the same order of magnitude in~\eqref{DEEBNeqn1} as $j$ tends infinity; this corresponds
to some kind of coefficient extraction with respect to $[\mathbf{X}^{j(\mathbf{n})}]$.
We obtain the for $\mathbf{n}=(n_1,\dots,n_d)\in\N_0^{d}\setminus\{\boldsymbol{0},\mathbf{e}_1,\dots,\mathbf{e}_d\}$ the recurrence relation
\begin{equation}
\label{DEEBNallg1}
\alpha_{\mathbf{n}}\Big(-\frac{1}{zT^{2d}}+\sum_{\ell=-d}^{d}\mathbf{X}^{\ell\mathbf{n}}\Big)
=\sum_{\ell=2}^{2d+1}(-1)^{\ell}\sum_{\sum_{k=1}^{\ell}\mathbf{g}_k=\mathbf{n}}\bigg(\prod_{k=1}^{\ell}\alpha_{\mathbf{g}_k}\bigg)
\sum_{\mathbf{b}_{\ell}\subseteq \{-d,\dots,d\}}\bigg(\prod_{k=1}^{\ell}\mathbf{X}^{b_k\mathbf{g}_k}\bigg);
\end{equation}
here $\mathbf{g}_k\in\N_0^{d}\setminus\{\boldsymbol{0}\}$ denotes a vector of length $d$, for $1\le k\le \ell $ and $2\le \ell\le 2d+1$, $\mathbf{b}_{\ell}=\{b_1,\dots,b_{\ell}\}$ runs over all subset of $\{-d,\dots,d\}$ of size $\ell$, $2\le \ell\le 2d+1$. Let $f(\mathbf{w})$ denote the formal power series $f(\mathbf{w})=\sum_{\mathbf{n}\ge \boldsymbol{0}}\alpha_{\mathbf{n}}\mathbf{w}^{\mathbf{n}}$. The recurrence relation above for $\alpha_{\mathbf{n}}=\alpha_{\mathbf{n}}(\mathbf{X})$ is equivalent to a functional equation
for $f(\mathbf{w})$
\begin{equation}
\label{DEEBNallg2}
\frac{-1}{zT^{2d}}f(\mathbf{w})+\sum_{\ell=-d}^{d}f(\mathbf{X}^{\ell}\mathbf{w})=\sum_{\ell=2}^{2d+1}(-1)^{\ell}
\sum_{\mathbf{b}_{\ell}\subseteq \{-d,\dots,d\}} \bigg(\prod_{k=1}^{\ell}f(\mathbf{X}^{b_k}\mathbf{w})\bigg);
\end{equation}
here we use the notation $\mathbf{X^{\ell}w}:=(X_1^{\ell}w_1,X_2^{\ell}w_2,\dots,X_d^{\ell}w_d)$. A \emph{crucial step} towards solving recurrence relation~\eqref{DEEBNallg1} is to study the recurrence relation $\alpha_{\mathbf{n}}$ with $\mathbf{n}=n_k\mathbf{e}_k=(0,\dots,0,n_k,0,\dots,0)$ for $1\le k\le d$.
Note that the simplified recurrence relation involves only terms $\alpha_{\mathbf{n}}$ such that $n_{\ell}=0$ for $\ell\neq k$.
Once these one parameter solutions are obtained, the general solution is immediately determined by equation~\eqref{DEEBNallg1}. Moreover,
by definition the formal power series $\rho_j(\mathbf{X})$ is a solution of the main equation~\eqref{DEEBNmaster}.

\subsection{One parameter solution\label{DEEBNsubseconepar}}
A solution with only one free parameter can be obtained by using a simpler ansatz using only one of the series $X_1(z),\dots,X_d(z)$.
Equivalently, we consider $\alpha_{\mathbf{n}}=\alpha_{\mathbf{n}}(\mathbf{X})$, with $\mathbf{n}=(0,\dots,0,n_k,0,\dots,0)$ for $1\le k\le d$.
We obtain the following simple solution.
\begin{lemma}
\label{DEEBNlem5}
Let $X=X(z)$ denote any of the $d$ small solutions $X_1(z),\dots,X_d(z)$ of the characteristic equation
$1=zT^{2d}\sum_{\ell=-d}^{d}X^{\ell}$. The recurrence relation~\eqref{DEEBNrec1a}
admits a solution with free parameter $\lambda$
\begin{equation*}
T_j=T\cdot\frac{(1-\lambda X^{d+1+j})(1-\lambda X^{2d+3+j})}{(1-\lambda X^{d+2+j})(1- \lambda X^{2d+2+j})},\quad j\in\Z.
\end{equation*}
\end{lemma}
\begin{proof}
Since by definition the series $X$ satisfies $1/(zT^{2d})=\sum_{\ell=-d}^{d}X^{\ell}$, the recurrence relation~\eqref{DEEBNallg1} simplifies to
\begin{equation*}
\alpha_{n+1}\Big(-\sum_{\ell=-d}^{d}X^{\ell}+\sum_{\ell=-d}^{d}X^{\ell(n+1)}\Big)
=\sum_{\ell=2}^{2d+1}(-1)^{\ell}\sum_{\sum_{k=1}^{\ell}g_k=n+1}\bigg(\prod_{k=1}^{\ell}\alpha_{g_k}\bigg)
\sum_{\mathbf{b}_{\ell}\subseteq \{-d,\dots,d\}}\bigg(\prod_{k=1}^{\ell}X^{b_kg_k}\bigg),\quad n\ge 1.
\end{equation*}
By experiments with \texttt{Maple} we obtain a solution with free parameter $\alpha_1$,
\begin{equation}
\label{DEEBNoneparsol}
\alpha_n= \frac{\alpha_1^n X^{n-1}(1-X^{nd})}{(1-X^d)(1-X)^{n-1}(1-X^{d+1})^{n-1}},\quad n\ge 1.
\end{equation}
One can checks for small $d=1,2,3,\dots$ that the arising function $f(w)=\sum_{n\ge1}\alpha_{n}w^n$ satisfies the functional equation
\begin{equation*}
-f(w)\sum_{\ell=-d}^{d}X^{\ell}+\sum_{\ell=-d}^{d}f(X^{\ell}w)=\sum_{\ell=2}^{2d+1}(-1)^{\ell}
\sum_{\mathbf{b}_{\ell}\subseteq \{-d,\dots,d\}} \bigg(\prod_{k=1}^{\ell}f(X^{b_k}w)\bigg).
\end{equation*}
Now we set $\alpha_1=\lambda X^{d+1}(1-X)(1-X^{d+1})$ in order to simplify the calculations. We get $T_j=T(1-\sum_{n\ge 1}\alpha_n X^{jn}$;
one readily checks that the stated solution satisfies recurrence relation~\eqref{DEEBNrec1a}.
\end{proof}

\subsection{The general solution}
An immediate application of Lemma~\ref{DEEBNlem5} and the explicit result~\eqref{DEEBNoneparsol} for $\alpha_n$ is the following result.
\begin{prop}
\label{DEEBNprop1}
Let $\rho_j=\sum_{\mathbf{n}\ge \boldsymbol{0}}\alpha_{\mathbf{n}}\mathbf{X}^{j\mathbf{n}}$, with $\rho_j=\rho_j(z)=\rho_j(\mathbf{X})$, and coefficients
$\alpha_{\mathbf{n}}=\alpha_{\mathbf{n}}(\mathbf{X})$ given by
\begin{equation*}
\alpha_{\mathbf{n}}=
\begin{cases}
0 &\text{for}\,\, \mathbf{n}=(0,0,\dots,0),\\
\displaystyle{\frac{\alpha_{\mathbf{e}_{\ell}}^{n_{\ell}} X_{\ell}^{n_{\ell}-1}(1-X_{\ell}^{n_{\ell}d})}{(1-X_{\ell}^d)(1-X_{\ell})^{n_{\ell}-1}(1-X_{\ell}^{d+1})^{n_{\ell}-1}}}
&\text{for}\,\, \mathbf{n}=n_{\ell} \mathbf{e}_{\ell},\quad\text{for}\,\, n_{\ell}\ge 1,\quad 1\le \ell \le d,\\[0.3cm]
\text{determined by recurrence relation~\eqref{DEEBNallg1}} &\text{for}\,\,\mathbf{n}\in\N_0^{d}\setminus\{\boldsymbol{0},\mathbf{e}_1,\dots,\mathbf{e}_d\},
\end{cases}
\end{equation*}
with unspecified initial values $\alpha_{\mathbf{e}_{\ell}}$, $1\le \ell\le d$. Then, the formal power series $T_j=T(1-\rho_j)$ satisfies the recurrence relation
\begin{equation*}
T_{j}(z)=1+z\prod_{\ell=-d}^{d}T_{j+\ell}(z).
\end{equation*}
Equivalently, the formal power series $\rho_j$ satisfies the equation
\begin{equation*}
\rho_j=zT^{2d}\sum_{\ell=1}^{2d+1}(-1)^{\ell-1}\sum_{\mathbf{b}_{\ell}\subseteq \{-d,\dots,d\}}\bigg(\prod_{k=1}^{\ell}\rho_{j+b_k}\bigg).
\end{equation*}
\end{prop}
\begin{proof}
By the recursive description~\eqref{DEEBNallg1} of $\mathbf{\alpha}_{\mathbf{n}}$ and the simple observation that for $z$ near zero we have
\begin{equation*}
\frac{1}{-\frac{1}{zT^{2d}}+\sum_{\ell=-d}^{d}\mathbf{X}^{\ell\mathbf{n}}}
%=\frac{-zT^{2d}\mathbf{X}^{d\mathbf{n}}}{1-zT^{2d}\sum_{\ell=0}^{2d}\mathbf{X}^{\ell\mathbf{n}}}
=-zT^{2d}\mathbf{X}^{d\mathbf{n}}\sum_{k\ge 0} \Big(zT^{2d}\sum_{\ell=0}^{2d}\mathbf{X}^{\ell\mathbf{n}}\Big)^k,
\end{equation*}
the values $\mathbf{\alpha}_{\mathbf{n}}=\mathbf{\alpha}_{\mathbf{n}}$ can be written as formal power series in $\tilde{T}=\tilde{T}(z)=T-1$,
according to $zT^{2d}=1-T=-\tilde{T}$, with $\tilde{T}(0)=0$, and $\mathbf{X}$. Consequently, by definition of the values $\mathbf{\alpha}_{\mathbf{n}}$~\eqref{DEEBNallg1} the left and right hand side
\eqref{DEEBNmaster} coincide.
\end{proof}

The huge obstacle concerning our enumeration problem~\eqref{DEEBNrec1a} is to explicitly determine the values
$\alpha_{\mathbf{n}}$ in the general case in order to adapt the initial values $\lambda_{\ell}$, $1\le \ell\le d$, to the initial conditions
$T_{-\ell}=1$, $1\le \ell\le d$. The only way known to us to obtain $\alpha_{\mathbf{n}}$
is either guessing the solution after experiments, or to solve the functional equation~\eqref{DEEBNallg2}.
Unfortunately, we do not know how to directly solve~\eqref{DEEBNallg2} and we did not manage yet to guess a general formula for $\alpha_{\mathbf{n}}$.

\section{Embedded 2d-ary trees with small labels}
The considerations for embedded $2d$-ary trees are similar to $(2d+1)$-ary trees, therefore we will be more brief.
According to the ansatz $T_j=T(1-\rho_j)$ we expend Equation~\ref{DEEBNrec1a} and obtain the equation
\begin{equation}
\label{DEEBNmaster2d}
\rho_j=-zT^{2d-1}+zT^{2d-1}\prod_{\ell=1}^{d}(1-\rho_{j+2\ell-1})(1-\rho_{j-2\ell+1})=
zT^{2d-1}\sum_{\ell=1}^{2d}(-1)^{\ell}\sum_{\mathbf{b}_{\ell}\subseteq B_{2d}}\bigg(\prod_{k=1}^{\ell}\rho_{j+b_k}\bigg),
\end{equation}
with $\mathbf{b}_{\ell}=\{b_1,\dots,b_{\ell}\}$ running over all subset of $B_{2d}=\{-(2d-1),-(2d-3),\dots,2d-1\}$ of size $\ell$, $1\le \ell\le 2d$.
Comparing the terms tending at a similar rate to zero as $j$ tends to infinity we obtain the linear recurrence relation
\begin{equation*}
\rho_j=zT^{2d-1}\sum_{\ell=1}^{d}(\rho_{j+2\ell-1}+\rho_{j-2\ell+1}).
\end{equation*}
An ansatz $\rho_j=\alpha X^j$, assuming that there exists a formal power series $X=X(z)$ with $|X|<1$ for $z$ near 0, leads to the characteristic equation
\begin{equation}
\label{DEEBNchar2d}
1=zT^{2d-1}\sum_{\ell=1}^{d}\big(X^{2\ell-1}+X^{-2\ell+1}\big),\quad\text{or equivalently}\quad 1=Z\sum_{\ell=1}^{d}\big(X^{2\ell-1}+X^{-2\ell+1}\big),\quad \text{with}\,\,
Z:=zT^{2d-1}.
\end{equation}
We apply Lemma~\ref{DEEBNlemFla1} to case $\mathcal{S}=\{2\ell-1,-2\ell+1 \mid 1\le \ell \le d\}$, with weights all equal to one, and $Z=zT^{2d-1}$.
Consequently the equation $1=Z\sum_{\ell=1}^{d}\big(X^{2\ell-1}+X^{-2\ell+1}\big)$, with $Z=zT^{2d-1}$, has $2d-1$ small solution $X_1(z),\dots X_{2d-1}(z)$. As before, we refine the previous ansatz $\rho_j=\alpha X^j$ in terms of the $2d-1$ solutions $X_1(z),\dots X_{2d-1}(z)$ of the characteristic equation in the following way.
\begin{equation*}
\rho_j=\sum_{n_1,\dots,n_{2d-1}\ge 0}\alpha_{n_1,\dots,n_{2d-1}}X_1^{jn_1}\dots X_{2d-1}^{jn_{2d-1}}
=\sum_{\mathbf{n}\ge \boldsymbol{0}}\alpha_{\mathbf{n}}\mathbf{X}^{j\mathbf{n}},
\end{equation*}
with $\alpha_{0,0,\dots,0}=0$ and unspecified initial values $\alpha_{\mathbf{e}_{\ell}}$, $1\le \ell\le 2d-1$, where $\mathbf{e}_{\ell}$ denotes the $\ell$-th unit vector. According to the refined ansatz the equation~\eqref{DEEBNmaster2d} we obtain the for $\mathbf{n}=(n_1,\dots,n_{2d-1})\in\N_0^{2d-1}\setminus\{\boldsymbol{0},\mathbf{e}_1,\dots,\mathbf{e}_{2d-1}\}$ the recurrence relation
\begin{equation}
\label{DEEBNallg2d}
\alpha_{\mathbf{n}}\Big(-\frac{1}{zT^{2d-1}}+\sum_{\ell=1}^{d}\big(\mathbf{X}^{(2\ell-1)\mathbf{n}}+\mathbf{X}^{(-2\ell+1)\mathbf{n}}\big)\Big)
=\sum_{\ell=2}^{2d}(-1)^{\ell}\sum_{\sum_{k=1}^{\ell}\mathbf{g}_k=\mathbf{n}}\bigg(\prod_{k=1}^{\ell}\alpha_{\mathbf{g}_k}\bigg)
\sum_{\mathbf{b}_{\ell}\subseteq B_{2d}}\bigg(\prod_{k=1}^{\ell}\mathbf{X}^{b_k\mathbf{g}_k}\bigg);
\end{equation}
here $\mathbf{g}_k\in\N_0^{2d-1}\setminus\{\boldsymbol{0}\}$ denotes a vector of length $2d-1$, for $1\le k\le \ell $ and $2\le \ell\le 2d$, $\mathbf{b}_{\ell}=\{b_1,\dots,b_{\ell}\}$ runs over all subset of $B_{2d}=\{-(2d-1),-(2d-3),\dots,2d-1\}$ of size $\ell$, $1\le \ell\le 2d$.

\subsection{One parameter solution\label{DEEBNsubseconepar2d}}
The solutions with one free parameter can be obtained by using a simpler ansatz using only one of the series $X_1(z),\dots,X_{2d-1}(z)$.
Equivalently, we consider $\alpha_{\mathbf{n}}=\alpha_{\mathbf{n}}(\mathbf{X})$, with $\mathbf{n}=(0,\dots,0,n_k,0,\dots,0)$ for $1\le k\le 2d-1$.
We obtain the following simple solution.
\begin{lemma}
\label{DEEBNlem6}
Let $X=X(z)$ denote any of the $2d-1$ small solutions $X_1(z),\dots,X_{2d-1}(z)$ of the characteristic equation
$1=zT^{2d-1}\sum_{\ell=1}^{d}(X^{2\ell-1}+X^{-2\ell+1})$. The recurrence relation~\eqref{DEEBNrec1b}
admits a solution with free parameter $\lambda$
\begin{equation*}
T_j=T\cdot\frac{(1-\lambda X^{d+1+j})(1-\lambda X^{3d+4+j})}{(1-\lambda X^{d+3+j})(1- \lambda X^{3d+2+j})},\quad j\in\Z.
\end{equation*}
\end{lemma}
\begin{proof}
By experiments with \texttt{Maple} we obtain a solution with free parameter $\alpha_1$ of the simplified recurrence relation~\eqref{DEEBNallg2d}
\begin{equation}
\label{DEEBNoneparsol2d}
\alpha_n= \frac{\alpha_1^n X^{2(n-1)}(1-X^{n(2d-1)})}{(1-X^{2d-1})(1-X^2)^{n-1}(1-X^{2d+1})^{n-1}},\quad n\ge 1.
\end{equation}
Now we set $\alpha_1=\lambda X^{d+1}(1-X^2)(1-X^{2d+1})$ in order to simplify the calculations; one readily checks that the stated solution satisfies recurrence relation~\eqref{DEEBNrec1b}.
\end{proof}

\subsection{The general solution}
An immediate application of Lemma~\ref{DEEBNlem6} and the explicit result~\eqref{DEEBNoneparsol2d} for $\alpha_n$ is the following result.
\begin{prop}
\label{DEEBNprop2}
Let $\rho_j=\sum_{\mathbf{n}\ge \boldsymbol{0}}\alpha_{\mathbf{n}}\mathbf{X}^{j\mathbf{n}}$, with $\rho_j=\rho_j(z)=\rho_j(\mathbf{X})$, and coefficients
$\alpha_{\mathbf{n}}=\alpha_{\mathbf{n}}(\mathbf{X})$ given by
\begin{equation*}
\alpha_{\mathbf{n}}=
\begin{cases}
0 &\text{for}\,\, \mathbf{n}=(0,0,\dots,0),\\
\displaystyle{\frac{\alpha_{\mathbf{e}_{\ell}}^n X_{\ell}^{2(n_{\ell}-1)}(1-X_{\ell}^{n_{\ell}(2d-1)})}{(1-X_{\ell}^{2d-1})(1-X_{\ell}^2)^{n_{\ell}-1}(1-X_{\ell}^{2d+1})^{n_{\ell}-1}}}
&\text{for}\,\, \mathbf{n}=n_{\ell} \mathbf{e}_{\ell},\quad\text{for}\,\, n_{\ell}\ge 1,\quad 1\le \ell \le 2d-1,\\[0.3cm]
\text{determined by recurrence relation~\eqref{DEEBNallg2d}} &\text{for}\,\,\mathbf{n}\in\N_0^{2d-1}\setminus\{\boldsymbol{0},\mathbf{e}_1,\dots,\mathbf{e}_{2d-1}\},
\end{cases}
\end{equation*}
with unspecified initial values $\alpha_{\mathbf{e}_{\ell}}$, $1\le \ell\le 2d-1$. Then, the formal power series $T_j=T(1-\rho_j)$ satisfies the recurrence relation
\begin{equation*}
T_{j}(z)=1+z\prod_{\ell=1}^{d}\big(T_{j+2\ell-1}(z)+T_{j-2\ell+1}(z)\big).
\end{equation*}
\end{prop}

\section{Enumeration of lattice paths and degenerated embedded trees\label{Flajo}}
Fix a set of step vectors $\mathcal{S}=\{(1,b_1),\dots, (1,b_m)\}$ with $b_{\ell}\in\Z$, $1\le \ell\le m$.
A simple lattice path, also called a walk, is a sequence $(v_1,\dots, v_n)$ such that for each $v_i\in\mathcal{S}$.
A meander is a simple lattice path restricted to $\N_0\times\N_0$. An excursion is a meander with starting point and end point on the $y$-axis.
It is often useful to consider weighted lattice paths $\Pi=\{w_1,\dots, w_m\}$, where weights $w_{\ell}$ is associated
to step $(1,b_{\ell})$. The weight of a path is defined as the product of the weight of the steps.

\begin{figure}[htb]
\centering
\includegraphics[angle=0,scale=0.7]{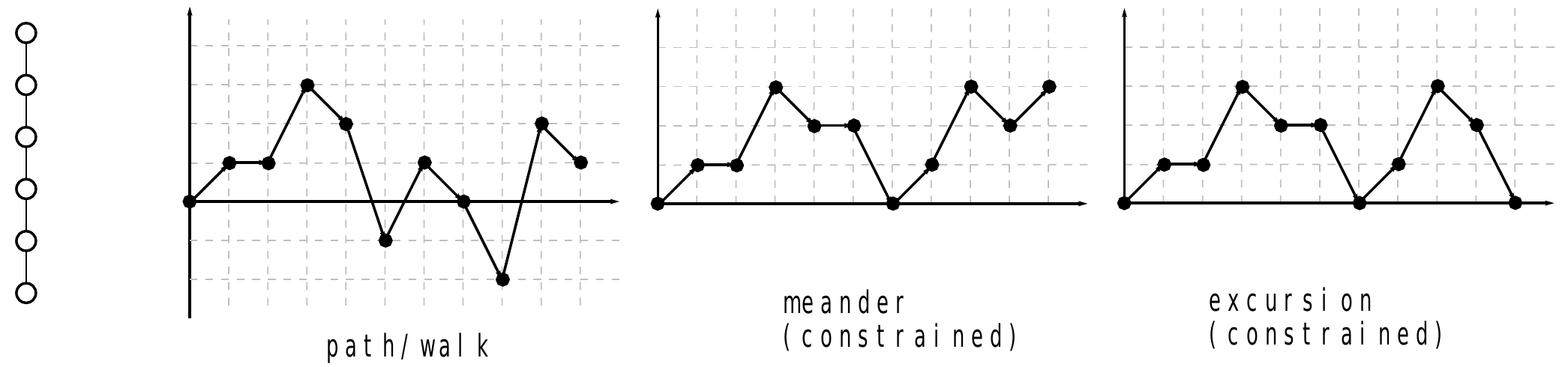}
\caption{Degenerated 1-ary tree defined by $T(z)=1+zT(z)$, and three types of lattice paths: unrestricted paths, meanders restricted to $\N_0\times\N_0$, and excursions starting and ending at level zero.}
\label{DEEBNfig2}
\end{figure}

Banderier and Flajolet~\cite{BandFla2002}, amongst many other things, derived the generating functions of meanders and excursions with respect to a set of step vectors $\mathcal{S}$ and weights $\Pi$ using the kernel method. In the following we will rederive (and slightly refine) these generating functions using the method of Section~\ref{Franc}. Following~\cite{BandFla2002} we introduce the characteristic polynomial
$P(X):=\sum_{\ell=1}^{m}w_{\ell}X^{b_{\ell}}$ of step set $\mathcal{S}=\{(1,b_1),\dots, (1,b_m)\}$, with $c=-\min\{b_\ell\}$ and $d=\max\{b_\ell\}$, and weights $\Pi=\{w_1,\dots, w_m\}$.
Let $T_j(z)$ denote the generating functions of all weighted meanders with steps $\mathcal{S}=\{(1,b_1),\dots, (1,b_m)\}$ and weights $\Pi=\{w_1,\dots, w_m\}$ starting at level $j$, with $j\ge 0$. We have the infinite system of recurrence relations
\begin{equation}
\label{DEEBNlattice1}
T_j(z)=1+z\sum_{\ell=1}^{m}w_{\ell}T_{j+b_{\ell}}(z),\quad j\ge 0,
\end{equation}
with initial conditions $T_{-1}(z)=\cdots=T_{-c}(z)=0$. This
recurrence relation can be interpreted as a recurrence relation for
embedded trees with respect to the class of degenerated unary trees defined by the equation $T(z)=1+zT(z)$.

 \smallskip

For $j\to\infty$ we have convergence in the sense of formal power series
$T_j(z)\to T(z)$, where $T(z)$ is the generating function of unconstrained lattice paths,
i.e.~starting at zero and ending anywhere, with steps $\mathcal{S}=\{(1,b_1),\dots, (1,b_m)\}$ and weights $\Pi=\{w_1,\dots, w_m\}$,
\begin{equation*}
T(z)=1+z\sum_{\ell=1}^{m}w_{\ell}T(z),\quad\text{or}\quad T(z)=\frac{1}{1-zP(1)},
\end{equation*}
where $P(X)=\sum_{\ell=1}^{m}w_{\ell}X^{b_{\ell}}$ denotes the characteristic polynomial of steps and weights. We use the ansatz $T_j(z)=T(z)(1-\rho_j(z))$ to obtain
\begin{equation*}
T(z)(1-\rho_j(z))=1+z\sum_{\ell=1}^{m}w_{\ell}T(z)\big(1-\rho_{j+b_{\ell}}(z)\big).
\end{equation*}
Consequently, we get a linear recurrence relation for $\rho_j(z)$,
\begin{equation*}
\rho_j(z)=z\sum_{\ell=1}^{m}w_{\ell}\rho_{j+b_{\ell}}(z).
\end{equation*}
Setting $\rho_j(z)=X^j$ we obtain after simple manipulations the characteristic equation
\begin{equation}
\label{HiFla}
1-z\sum_{\ell=1}^{m}w_{\ell}X^{b_{\ell}}=0,\quad\text{or equivalently},\quad 1-zP(X)=0.
\end{equation}
By Lemma~\ref{DEEBNlemFla1} there exist $c$ solution small solutions
$X_1(z),\dots,X_c(z)$ for $z$ in a neighborhood of zero, and the
general solution is given by
\begin{equation*}
\rho_j(z)=\sum_{\ell=1}^{c}\alpha_{\ell}X_{\ell}^j,
\end{equation*}
with unspecified $\alpha_{\ell}$, $1\le \ell\le c$.
Consequently, we obtain the following result. 
\begin{theorem}
A family of solutions of the system of recurrence relations~\eqref{DEEBNlattice1}, well defined for $z$ in a neighbourhood of zero, is given in terms of the $c$ small solutions $X_1,\dots,X_c$ of the characteristic equation~\eqref{HiFla},
\begin{equation*}
T_j(z)= T(z) (1-\sum_{\ell=1}^{c}\alpha_{\ell}X_{\ell}^j),
\end{equation*}
with free parameters $\alpha_{\ell}$, with $1\le \ell\le c$. The parameters $\alpha_{\ell}$
are independent of $j$, but may depend on $X_1,\dots,X_c$ and $z$.
\end{theorem}
Adapting to the initial conditions of the meanders $T_{-1}(z)=\cdots=T_{-c}(z)=0$ lead to a system of $c$ linear equations
\begin{equation*}
 1-\sum_{\ell=1}^{c}\alpha_{\ell}X_{\ell}^{-i}=0,\quad\text{for}\quad 1\le i\le c.
\end{equation*}
This system is easily solved using Cramer's rule and Vandermonde's determinant; we obtain the result
\begin{equation*}
\alpha_{\ell}= \frac{\Big(\prod_{1\le i<k\le c}(X_i-X_k)\Big)\Big|_{X_{\ell}=1}}{\prod_{1\le i<k\le c}(X_i-X_k)}\cdot X_{\ell}^c.
\end{equation*}
Consequently, we get after simple manipulations the following result.
\begin{coroll}
The generating function of meanders with steps $\mathcal{S}=\{(1,b_1),\dots, (1,b_m)\}$, weights $\Pi=\{w_1,\dots, w_m\}$
and characteristic polynomial $P(X):=\sum_{\ell=1}^{m}w_{\ell}X^{b_{\ell}}$ starting at level $j\ge 0$ is given by
\begin{equation*}
T_j(z)=T(z)\bigg(1-\sum_{\ell=1}^{c}\frac{\Big(\prod_{1\le i<k\le c}(X_i-X_k)\Big)\Big|_{X_{\ell}=1}}{\prod_{1\le i<k\le c}(X_i-X_k)}\cdot X_{\ell}^{c+j}\bigg)=
\frac{1}{1-zP(1)}\sum_{f=0}^{j}h_{f}(X_1,\dots,X_c)\prod_{\ell=1}^{c}(1-X_{\ell}),
\end{equation*}
where the $h_{f}(X_1,\dots,X_c)=\sum_{1\le i_1\le\dots \le i_f\le c}\prod_{\ell=1}^{f}X_{i_{\ell}}$ are the complete homogeneous symmetric polynomials of degree $f$ in $X_1,\dots,X_c$, denoting the small solutions
of the characteristic equation~\eqref{HiFla}.
\end{coroll}
Note that the complete homogeneous symmetric polynomials satisfy the formal power series identity
\begin{equation*}
\sum_{f\ge 0}h_f(X_1,\dots,X_c)t^f=\prod_{\ell=1}^{c}\frac{1}{1-X_{\ell}t},
\end{equation*}
and thus we indeed have $T_j(z)\to T(z)=1/(1-P(z))$, for $j$ tending to infinity. In order to fix the endpoint of the considered lattice paths we introduce an additional variable $v$ encoding the steps $\mathcal{S}=\{(1,b_1),\dots, (1,b_m)\}$.
This leads to a refined recurrence relation with respect to the refined characteristic polynomial
\begin{equation*}
P(Xv)= \sum_{\ell=1}^{m}w_{\ell}\big(vX)^{b_{\ell}}.
\end{equation*}
The $c$ small solutions of the characteristic equation $1-z
\sum_{\ell=1}^{m}w_{\ell}\big(vX)^{b_{\ell}}=0$ are given by the
previously encountered (case $v=1$) solutions $X_1(z),\dots,X_c(z)$
for $z$ in a neighborhood of zero divided by $v$. Consequently, we
obtain the following result.
\begin{coroll}
The generating function of meanders with steps $\mathcal{S}=\{(1,b_1),\dots, (1,b_m)\}$, weights $\Pi=\{w_1,\dots, w_m\}$
and characteristic polynomial $P(X):=\sum_{\ell=1}^{m}w_{\ell}X^{b_{\ell}}$ starting at level $j\ge 0$, where $v$ marks the level of the endpoint of the path, is given by
\begin{equation*}
T_j(z,v)=\frac{1}{1-zP(v)}\sum_{f=0}^{j}h_{f}(\frac{X_1}{v},\dots,\frac{X_c}{v})\prod_{\ell=1}^{c}(1-\frac{X_{\ell}}{v}),
\end{equation*}
where the $h_{f}(X_1,\dots,X_c)$ are the complete homogeneous symmetric polynomials in $X_1,\dots,X_c$, denoting the small solutions
of the characteristic equation~\eqref{HiFla}.
\end{coroll}
\begin{remark}
We reobtain the result of Banderier and Flajolet~\cite{BandFla2002} for the enumeration of meanders by setting $j=0$.
Furthermore, the generating function of excursions, starting at level $j$ and ending at level $j$ never going below
zero is obtained by extracting the coefficient $[v^0]T_j(z,v)$. In particular, one can reobtain the generating function of the number of excursions starting and ending at level zero.
\end{remark}

\section{Vicious walkers, osculating walkers, and two walks in the quarter plane}
We will show how the method of Section~\ref{Franc} can be used to (re-)derive some
results concerning so-called vicious walkers and osculating walkers, and also concerning two walks confined in the quarter plane.

\subsection{Vicious walkers and osculating walkers in the lock step model} 
We consider three walkers each equipped with steps $\mathcal{S}=\{(1,1),(1,-1)\}$. In the lock step model the walkers move simultaneously at each discrete time step such that the walkers never cross. A walker configuration is called osculating if the walkers never share an edge, and vicious if the walkers never meet. We also consider the model of up-down walkers, configurations where the first and second walker may meet and share a down step, and the second and third walker may also meet and share an up step. For a comprehensive overview of this topic we refer the reader to the work of Bousquet-M\'elou~\cite{oscu} and the references therein. Following~\cite{oscu} we consider so-called $(i,j)$-stars, which are non-crossing configurations of walkers with initial level difference $2i$ between the first two walkers and level difference $2j$ between the second and third walker.
Let $T_{i,j}=T_{i,j}(z)=\sum_{n\ge 0}w_{n;i,j}z^n$ denote the generating function of the number $w_{n;i,j}$ of $(i,j)$-star walker configurations in the lock step model of length $n$. According to the step set $\mathcal{S}$ we obtain the system of recurrence relations
\begin{equation}
\label{DEEBNoscu1}
T_{i,j}=1+z\big(2T_{i,j}+T_{i+1,j}+T_{i,j+1}+T_{i-1,j+1}+T_{i,j-1}+T_{i+1,j-1}+T_{i-1,j}\big),\quad\text{for}\quad i,j\ge 1
\end{equation}
with initial conditions $T^{[\mathcal{V}]}_{i,0}=T^{[\mathcal{V}]}_{0,j}=0$, $i,j\ge 0$ for vicious
walkers, $T^{[\mathcal{O}]}_{i,0}=1+z(T_{i,1}+T_{i-1,1})$,
$T^{[\mathcal{O}]}_{0,j}=1+z(T_{1,j}+T_{1,j-1})$, $i,j\ge 1$ for osculating walkers and conditions
$T^{[\mathcal{U}]}_{i,0}=1+z(T_{i,0}+T_{i+1,0}+T_{i,1}+T_{i-1,1})$, $T^{[\mathcal{U}]}_{0,j}=1+z(T_{0,j}+T_{0,j+1}+T_{1,j}+T_{1,j-1})$ for up-down walkers. 

\smallskip 

It turns out to be beneficial to study the slightly more general system of recurrence relations 
\begin{equation}
\label{DEEBNgen1}
T_{i,j}=1+z\big(wT_{i,j}+T_{i+1,j}+T_{i,j+1}+T_{i-1,j+1}+T_{i,j-1}+T_{i+1,j-1}+T_{i-1,j}\big),\quad\text{for}\quad i,j\ge 1,
\end{equation}
with weight $w\ge 0$. For $i,j\to\infty$ we have $T_{i,j}\to T$ in the sense of formal power series, with
\begin{equation*}
T=1+(w+6)zT,\quad \text{or equivalently}\quad T=\frac{1}{1-(w+6)z}.
\end{equation*}
Consequently, we set $T_{i,j}=T(1-\rho_{i,j})$ with $\lim_{i,j\to\infty} \rho_{i,j}=0$.
By the ansatz and the definition of series $T$ we obtain the following recurrence relation for $\rho_{i,j}$.
\begin{equation*}
\rho_{i,j}= z \big(w\rho_{i,j}+\rho_{i+1,j}+\rho_{i,j+1}+\rho_{i-1,j+1}+\rho_{i,j-1}+\rho_{i+1,j-1}+\rho_{i-1,j}\big).
\end{equation*}
Since we assume that $\lim_{i,j\to\infty} \rho_{i,j}=0$, we use the ansatz
\begin{equation*}
\rho_{i,j}=\alpha\cdot X(z)^{i}+\beta\cdot Y(z)^{j}+\gamma \cdot Z(z)^{i+j},
\end{equation*}
with unspecified $\alpha,\beta,\gamma$ independent of $i$ and $j$,
assuming that $|X|<1$, $|Y|<1$, and $|Z|<1$ for $z$ in a
neighborhood of zero. It turns out that all series $X$, $Y$ and $Z$
satisfy the same characteristic equation,
\begin{equation*}
X=z\big(wX+X^2+X+1+X+X^2+1\big)=z(2+(2+w)X+2X^2),\quad X=Y=Z.
\end{equation*}
We obtain the following result. 
\begin{theorem}
\label{DEEBNthewalks}
A family of solutions the system of recurrence relations~\eqref{DEEBNgen1}, being well defined for $z$ in a neighbourhood of zero, is given in terms of the function $X=X(z)$, defined by 
\begin{equation*}
X=z(2+(2+w)X+2X^2),\quad\text{with}\quad X=\frac{1-z(2+w)-\sqrt{(1-z(2+w))^2-16z^2}}{4z},
\end{equation*}
in the following way.
\begin{equation*}
T_{i,j}=T\cdot\big(1-\alpha\cdot X^{i}-\beta\cdot X^{j}-\gamma \cdot X^{i+j}\big),\quad\text{with}\quad T=\frac{1}{1-z(w+6)},
\end{equation*}
with weight $w\ge 0$, and free parameters $\alpha,\beta,\gamma$ independent of $i$ and $j$.
\end{theorem}

Subsequently, we set $w=2$ and adapt Theorem~\ref{DEEBNthewalks} to the various initial conditions.
For vicious walkers we adapt to the initial conditions
\begin{equation*}
\begin{split}
T^{[\mathcal{V}]}_{0,j}&=T(1-\alpha-(\beta+\gamma)X^j)=0\quad j\ge 0,\\
T^{[\mathcal{V}]}_{i,0}&=T(1-\beta-(\alpha+\gamma)X^i)=0\quad i\ge 0.
\end{split}
\end{equation*}
We obtain the equations
\begin{equation*}
\beta=-\gamma,\quad \alpha=-\gamma,\quad \alpha=1,\quad \beta=1,
\end{equation*}
and obtain the proper solution $\alpha=\beta=-\gamma=1$.
For osculating walkers we adapt to the initial conditions $T^{[\mathcal{O}]}_{0,j}=1+z(T_{1,j}+T_{1,j-1})$ and $T^{[\mathcal{O}]}_{i,0}=1+z(T_{i,1}+T_{i-1,1})$,
leading to the system of equations
\begin{equation*}
\begin{split}
T(1-\alpha-(\beta+\gamma)X^j))&=1 + zT\big(1-\alpha X-\beta X^j -\gamma X^{j+1}\big) +  zT\big(1-\alpha X-\beta X^{j-1} -\gamma X^{j}\big)\quad j\ge 0,\\
T(1-\beta-(\alpha+\gamma)X^i)&=1 + zT\big(1-\alpha X^i -\beta X -\gamma X^{i+1}\big) +  zT\big(1-\alpha X^{i-1}-\beta X -\gamma X^{i}\big)\quad i\ge 0.
\end{split}
\end{equation*}
We obtain the equations
\begin{equation*}
\begin{split}
T(1-\alpha)&=1+2zT(1-\alpha X),\quad T(1-\beta)=1+2zT(1-\beta X),\\
(\alpha+\gamma)&=z\big(\alpha(1+\frac1X) + \gamma(1+X)\big),
(\beta+\gamma)=z\big(\beta(1+\frac1X) + \gamma(1+X)\big),
\end{split}
\end{equation*}
and easily obtain the proper solution $\alpha=\beta=3X/(1+2X)$, $\gamma=-3X/(2+X)$, using the relations $z=X/(2(1+X)^2)$ and $T=1/(1-8z)$.
For up-down-walkers we obtain the equations
\begin{equation*}
\begin{split}
T(1-\alpha)&=1+2zT(2-\alpha-\alpha X),\quad T(1-\beta)=1+2zT(2-\beta-\beta X),\\
(\alpha+\gamma)&=z\big(\alpha(2+X+\frac1X) + 2\gamma(1+X)\big),\quad
(\beta+\gamma)=z\big(\beta(2+X+\frac1X) + 2\gamma(1+X)\big),
\end{split}
\end{equation*}
and easily obtain the solutions $\alpha=\beta=2X/(1+X)$ and $\gamma=-X$.

\begin{coroll}[Bousquet-M\'elou~\cite{oscu}, Gessel]
The length generating function of three vicious walkers, three osculating walkers and three up-down walkers in $(i,j)$-star configuration is given in terms of
series $X$, given by
\begin{equation*}
X=2z(1+X)^2,\quad\text{or equivalently}\quad X=\frac{1-4z-\sqrt{1-8z}}{4z},
\end{equation*}
in the following way. For vicious walkers
\begin{equation*}
T^{[\mathcal{V}]}_{i,j}=T(1-X^i-X^j+X^{i+j})=\frac{1}{1-8z}\cdot (1-X^i)(1-X^j),
\end{equation*}
for osculating walkers
\begin{equation*}
T^{[\mathcal{O}]}_{i,j}=T(1-\frac{3X}{1+2X} \cdot X^i-\frac{3X}{1+2X}\cdot X^j+ \frac{3X}{2+X}\cdot X^{i+j})=\frac{1}{1-8z}(1-\frac{3 X^{i+1}}{1+2X}  -\frac{3X^{j+1}}{1+2X} + \frac{3X^{i+j+1}}{2+X}),
\end{equation*}
and for up-down walkers
\begin{equation*}
T^{[\mathcal{U}]}_{i,j}=T(1-\frac{2X}{1+X} \cdot X^i-\frac{2X}{1+X}\cdot X^j+ X\cdot X^{i+j})=\frac{1}{1-8z}(1-\frac{2 X^{i+1}}{1+X}  -\frac{2X^{j+1}}{1+X} + X^{i+j+1}).
\end{equation*}
\end{coroll}
One can easily obtain a refinement of the results above by counting
the number of osculations with variable $u$ and the number of shared
up and down edges by variable $w$, which leads to a refined of Theorem~\ref{DEEBNthewalks}. The results presented above
correspond to the special cases $(u,w)=(0,0)$, $(u,w)=(1,0)$ and $(u,v)=(1,1)$, respectively.

\begin{coroll}
The length generating function generating function of $(i,j)$ stars where $u$ counts the number of osculations and $w$ the number of shared up-down edges after osculations, is given by
\begin{equation*}
T_{i,j}(z,u,w)= T \Big(1 - \alpha(X^i+X^j) + \alpha\frac{2(1+X)-u(1+wX)}{2(1+X)-uX(1+w)}\cdot X^{i+j} \Big),
\end{equation*}
with $\alpha$ given by the following expression.
\begin{equation*}
\alpha=\frac{(1+X)^2-u(1-X+X^2+wX)}{(1+X)^2-uX(w+X)}.
\end{equation*}
\end{coroll}

\subsection{Vicious walkers and osculating walkers in the random turn model}
The random turn model of walkers is similar to the lock step model but at each discrete time step only one of the walkers is allowed to move. 
We consider three vicious walkers and osculating walkers with Dyck steps $\mathcal{S}=\{(1,1),(1,-1)\}$ or Motzkin steps 
$\mathcal{S}=\{(1,1),(1,-1),(1,0)\}$. Let $T_{i,j}=T_{i,j}(z)=\sum_{n\ge 0}w_{n;i,j}z^n$ denote the generating function of the total number of $w_{n;i,j}$ the number of walker configurations of length $n$, with initial level difference $i$ between the first two walkers and level difference $j$ between the second and third walker. For the step set $\mathcal{S}=\{(1,1),(1,-1)\}$ we obtain the system of recurrence relations 
\begin{equation}
\label{DEEBNturn1}
T_{i,j}=1+z\big(T_{i+1,j}+T_{i,j+1}+T_{i-1,j+1}+T_{i,j-1}+T_{i+1,j-1}+T_{i-1,j}\big),\quad\text{for}\quad i,j\ge 1
\end{equation}
with initial conditions $T^{[\mathcal{V}]}_{i,0}=T^{[\mathcal{V}]}_{0,j}=0$, $i,j\ge 0$ for vicious walkers,
and $T_{i,-1}=T_{-1,j}=0$ for osculating walkers. For the step set $\mathcal{S}=\{(1,1),(1,-1),(1,0)\}$ we obtain the system of recurrence relations 
\begin{equation}
\label{DEEBNturn2}
T_{i,j}=1+z\big(3T_{i+1,j}+T_{i,j+1}+T_{i-1,j+1}+T_{i,j-1}+T_{i+1,j-1}+T_{i-1,j}\big),\quad\text{for}\quad i,j\ge 1
\end{equation}
with initial conditions $T^{[\mathcal{V}]}_{i,0}=T^{[\mathcal{V}]}_{0,j}=0$, $i,j\ge 0$ for vicious walkers,
and $T^{[\mathcal{O}]}_{i,-1}=T^{[\mathcal{O}]}_{-1,j}=0$ for osculating walkers. Using Theorem~\ref{DEEBNthewalks} we obtain the following results.
\begin{coroll}
\label{DEEBNcorollDyck}
The length generating function of three vicious walkers and osculating walkers in the random turn model with Dyck steps $\mathcal{S}=\{(1,1),(1,-1)\}$ 
is given in terms of series $X$, defined by $X=2z(1+X+X^2)$. 
For vicious walkers we obtain $T^{[\mathcal{V}]}_{i,j}=\frac{1}{1-6z}\cdot (1-X^i)(1-X^j)$. 
The generating function of osculating walkers is given by 
$T^{[\mathcal{O}]}_{i,j}(z)=T^{[\mathcal{V}]}_{i+1,j+1}$. In 
particular, we obtain 
\begin{equation*}
T^{[\mathcal{O}]}_{0,0}(z)=T^{[\mathcal{V}]}_{1,1}(z)=\frac{1}{1-6z}\cdot (1-X)^2= \frac{1-2z-\sqrt{(1+2z)(1-6z)}}{8z^2}.
\end{equation*}
\end{coroll}
\begin{coroll}
The length generating function of three vicious walkers and osculating walkers in the random turn model with Motzkin steps $\mathcal{S}=\{(1,1),(1,-1),(1,0)\}$ 
is given in terms of series $X$, defined by $X=z(2+5X+2X^2)$. 
For vicious walkers we obtain $T^{[\mathcal{V}]}_{i,j}=\frac{1}{1-9z}\cdot (1-X^i)(1-X^j)$, and for osculating walkers we get
$T^{[\mathcal{O}]}_{i,j}=\frac{1}{1-9z}\cdot (1-X^{i+1})(1-X^{j+1})$. In 
particular, we obtain 
\begin{equation*}
T^{[\mathcal{O}]}_{0,0}(z)=T^{[\mathcal{V}]}_{1,1}(z)=\frac{1}{1-9z}\cdot (1-X)^2= \frac{1-5z-\sqrt{(1-z)(1-9z)}}{8z^2}.
%=1+5z+ 29z^2+185z^3+1257z^4+\dots;
\end{equation*}
%which is sequence \textbf{A059231}. 
\end{coroll}

\subsection{Two families of walks confined in the quarter plane}
Recently, Bousquet-M\'elou and Mishna~\cite{Marni} presented a
systematic approach to the problem of counting walks starting at the origin $(0,0)$ being confined to the quarter plane $\N_0\times\N_0$. In particular, for step sets $\mathcal{S}_1=\{(-1,0),(0,1),(1,-1)\}$ and
$\mathcal{S}_2=\{(-1,0),(0,1),(1,0),(0,-1),(-1,1),(1,-1)\}$ they
derived, amongst many other results, the generating function of
walks of length $n$ starting at the origin $(0,0)$, the result involves the generating function of the
Motzkin numbers. For the two models $\mathcal{S}_1$ and
$\mathcal{S}_2$ we will derive the generating function of the number
of paths starting at $(i,j)$ confined to the quarter plane
$\N_0\times\N_0$ using the approach of Bouttier, Di Francesco and
Guitter; note that the subsequently derived results can also be
obtained, in a more systematic manner, using the kernel
method~\cite{Marni}. Let $T_{i,j}=T_{i,j}(z)$ denote the generating
function of the total number of paths starting at point
$(i,j)\in\N_0\times\N_0$ confined to the quarter plane. Concerning
the first step set $\mathcal{S}_1=\{(-1,0),(0,1),(1,-1)\}$ we
obtain the system of recurrence relations for
$T_{i,j}=T_{i,j}^{[\mathcal{S}_1]}(z)$,
\begin{equation}
\label{DEEBNmarni1}
T_{i,j}(z)=1+z(T_{i-1,j}+T_{i,j+1}+T_{i+1,j-1}),\quad\text{for}\,\,i,j\ge 0,
\end{equation}
with initial values $T_{-1,j}=T_{i,-1}=0$ for $i,j\ge 0$.
We set $T_{i,j}=T(1-\rho_{i,j})$ with $\lim_{i,j\to\infty} \rho_{i,j}=0$,
and series $T$ given as the solution of the limiting equation $T=1+3zT$.
By the ansatz we obtain the following recurrence relation for $\rho_{i,j}$.
\begin{equation*}
\rho_{i,j}= z \big(\rho_{i-1,j}+\rho_{i,j+1}+\rho_{i+1,j}\big).
\end{equation*}
Since we assume that $\lim_{i,j\to\infty} \rho_{i,j}=0$, we use again the ansatz
\begin{equation*}
\rho_{i,j}=\alpha\cdot X(z)^{i}+\beta\cdot Y(z)^{j}+\gamma \cdot Z(z)^{i+j},
\end{equation*}
assuming that $|X|<1$, $|Y|<1$, and $|Z|<1$ for $z$ in a neighborhood of zero.
It turns out that all three series satisfy the same characteristic equation,
\begin{equation*}
X=z(1+X+X^2),\quad\text{or equivalently}\,\, X=\frac{1-z-\sqrt{(1+z)(1-3z)}}{2z},\quad\text{with}\,\, X=Y=Z.
\end{equation*}
Note that $X$ is the generating function of the Motzkin numbers. Hence, we obtain the family of solutions
\begin{equation*}
T_{i,j}= \frac{1}{1-3z}\Big(1-\alpha\cdot X^{i}-\beta\cdot X^{j} -\gamma \cdot X^{i+j} \Big).
\end{equation*}
Adapting to the initial conditions $T_{-1,j}=T_{i,-1}=0$, for $i,j\ge 0$, gives $
\alpha=\beta=X$ and $\gamma=-X^2$.
For the second step set $\mathcal{S}_2=\{(-1,0),(0,1),(1,0),(0,-1),(-1,1),(1,-1)\}$ we obtain the system of recurrence relations 
for $T_{i,j}=T_{i,j}^{[\mathcal{S}_s]}(z)$,
\begin{equation*}
T_{i,j}=1+z(T_{i-1,j}+T_{i+1,j}+T_{i,j+1}+T_{i,j-1}+T_{i-1,j+1}+T_{i+1,j-1}),\quad\text{for}\,\,i,j\ge 0,
\end{equation*}
with initial values $T_{-1,j}=T_{i,-1}=0$ for $i,j\ge 0$. This model is equivalent to the random turn model 
of three osculating Dyck walkers, and already solved in Corollary~\ref{DEEBNcorollDyck}.
\begin{theorem}
The length generating functions $T_{i,j}^{[\mathcal{S}_1]}(z)$ of paths confined to the quarter plane starting at $(i,j)\in\N_0\times\N_0$ with step set $\mathcal{S}_1=\{(-1,0),(0,1),(1,-1)\}$ is given by 
\begin{equation*}
T_{i,j}^{[\mathcal{S}_1]}(z)= \frac{1}{1-3z}\big(1-X_1^{i+1}\big)\big(1-X_1^{j+1}\big),\quad\text{with}\,\, X_1=z(1+X_1+X_1^2).
\end{equation*}
The length generating functions $T_{i,j}^{[\mathcal{S}_2]}(z)$ of paths confined to the quarter plane starting at $(i,j)\in\N_0\times\N_0$ with step set $\mathcal{S}_2=\{(-1,0),(0,1),(1,0),(0,-1),(-1,1),(1,-1)\}$ equals the length generating function 
of three osculating Dyck walkers in the random turn model, 
\begin{equation*}
T_{i,j}^{[\mathcal{S}_2]}(z)=T_{i,j}^{[\mathcal{S}_1]}(2z)= \frac{1}{1-6z}\big(1-X_2^{i+1}\big)\big(1-X_2^{j+1}\big),\quad\text{with}\,\, X_2=2z(1+X_2+X_2^2).
\end{equation*}
\end{theorem}

\section*{Conclusion}
We have shown that the ``asymptotic series method'' of Bouttier, Di Francesco and Guitter can be used to study
several families of embedded trees. Moreover, we used the method to study simple families of lattice path which can be considered as a degenerated family of embedded trees. Furthermore, we presented some other problems which can be solved using this method together with a suitable ansatz.

\end{document}